\documentclass[11pt]{amsart}
\usepackage{amsmath}
\usepackage{amsthm}
\usepackage{amsfonts}
\usepackage{amssymb}
\usepackage{mathtools}
\usepackage[english]{babel}
\usepackage{hyperref}

\newtheorem{thm}{Theorem}[section]
\newtheorem{lem}[thm]{Lemma}

\theoremstyle{definition}
 
\newcommand{\BQ}{{\mathbf{Q}}}

\newcommand{\floor}[1]{\left\lfloor#1\right\rfloor}

\newcommand{\acr}{\newline\indent}

\theoremstyle{remark}

\newtheorem{rem}[thm]{Remark}

\begin{document}
\title{Necessary conditions for binomial collisions}
\author[Tomohiro Yamada]{Tomohiro Yamada*}
\address{\llap{*\,}Center for Japanese language and culture\acr
                   Osaka University\acr
                   562-8558\acr
                   8-1-1, Aomatanihigashi, Minoo, Osaka\acr
                   JAPAN}
\email{tyamada1093@gmail.com}


\begin{abstract}
We shall give some necessary conditions for the equation $\binom{x}{a}=\binom{y}{b}$ to hold:
if $\binom{2n+\delta}{n-m}=\binom{2n+l}{n-k}$ with $\delta=0$ or $1$, $0<m\leq 0.735k$, $k<n$
and $n$ sufficiently large,
then $l>(cn/\log n)^{40/21}$ for some constant $c$.
\end{abstract}

\maketitle

\section{Introduction}\label{intro}

Several authors have studied integers appearing in Pascal's triangle more than once in a nontrivial way.
Clearly, we have $\binom{x}{0}=1$ for every nonnegative integer $x$,
$\binom{x}{k}=\binom{x}{x-k}$ and, writing $N=\binom{x}{k}$, $\binom{N}{1}=\binom{N}{N-1}=\binom{x}{k}$.
So that, we should consider the equation
\begin{equation}\label{eq11}
N=\binom{x}{a}=\binom{y}{b}
\end{equation}
in nonnegative integers $a, b, x, y$ with $2\leq a\leq x/2$ and $2\leq b\leq y/2$.
Moreover, we may assume that $x>y$, which implies $a<b$ since $\binom{x}{c}\geq \binom{x}{a}>\binom{y}{a}$ for $a\leq c\leq x/2$.
For example,
\begin{equation}
\begin{alignedat}{2}
\binom{16}{2} & =\binom{10}{3}=120, & \binom{21}{2} & =\binom{10}{4}=210, \\
\binom{56}{2} & =\binom{22}{3}=1540, & \binom{120}{2} & =\binom{36}{3}=7140, \\
\binom{153}{2} & =\binom{19}{5}=11628, & \binom{221}{2} & =\binom{17}{8}=24310, \\
\binom{78}{2} & =\binom{15}{5}=\binom{14}{6}=3003.
\end{alignedat}
\end{equation}

An infinite family has been found by Lind \cite{Lin} and rediscovered by Singmaster \cite{Sin} and Tovey \cite{Tov}.
Let $F_i$ be the $n$-th Fibonacci number, defined by $F_0=0, F_1=1$ and $F_{i+2}=F_{i+1}+F_i$
for $i=0, 1, 2, \ldots$.
Then, for every $i=0, 1, 2, \ldots$,
\begin{equation}
\binom{F_{2i+2}F_{2i+3}}{F_{2i}F_{2i+3}}=\binom{F_{2i+2}F_{2i+3}-1}{F_{2i}F_{2i+3}+1}.
\end{equation}
We note that this family includes $\binom{15}{5}=\binom{14}{6}$ but $\binom{78}{2}$ does not appear in this family.

Let $(a, b, x, y, N)$ be any further solution of \eqref{eq11}.
de Weger \cite{Weg} proved that any further solution of \eqref{eq11} must satisfy $a\geq 5, y>1000$
and $N>10^{30}$ and conjectured that \eqref{eq11} has no further solution.
Blokhuis, Brouwer and de Weger \cite{BBW} pushed de Weger's lower bounds up to $y>10^6$ an $N>10^{60}$.

Several pairs have been shown to never appear as $(a, b)$.
Avanesov \cite{Ava} showed that $(a, b)\neq (2, 3)$ and Pint\'{e}r showed that $(a, b)\neq (2, 4)$.
Mordell's result \cite{Mor} on the equation $Y(Y+1)=X(X+1)(X+2)$ implies that $(a, b)\neq (3, 4)$,
as de Weger \cite {Weg} pointed out.
Extending these results, Stroeker and de Weger \cite{SW} showed that $(a, b)\neq (2, 3), (2, 4), (2, 6), (2, 8), (3, 4), (3, 6), (4, 6), (4, 8)$.
These results can be obtained by determining all integer points on elliptic curves derived from \eqref{eq11}
for these values of $a, b$.
For other values of $a, b$, \eqref{eq11} will be difficult to solve since it yields curves genus $>1$.
Using recently developed techniques on hyperelliptic curves,
Bugeaud, Mignotte, Siksek, Stoll and Tengely \cite{BMSST} showed that $(a, b)\neq (2, 5)$.

An approach from the opposite direction will be the study of \eqref{eq11}
in the case $b$ is near to $y/2$.
Let $y=2n+\delta$ with $\delta=0$ or $1$ and $x=2n+l$ with $l>\delta$.
Moreover, let $m=n-b$ and $k=n-a$, so that $0\leq m<k<n/2$.
Now \eqref{eq11} becomes
\begin{equation}\label{eq12}
\binom{2n+\delta}{n-m}=\binom{2n+l}{n-k}.
\end{equation}

Now, we shall state our result.
\begin{thm}
If $n, m, l, k$ are integers satisfying \eqref{eq12} with 
$0\leq m<k<n/2$ and $m\leq 0.735k$, then $l>n(1.3132\log^2(2n)-2.00271)$.
Furthermore, if $c<0.68943$,
then $l>(cn/\log n)^{40/21}$ for sufficiently large $n$.
\end{thm}

Our argument would be generalized to show that, for any constants $\eta<1$ and $c<0.68943$, 
\eqref{eq12} has only finitely many solutions with $m\leq \eta k$ and $l<(cn/\log n)^{40/21}$.
Moreover, Cramer's conjecture that the next prime after $p$ is smaller than $p+c_1\log^2 p$
would allow us to replace $(cn/\log n)^{40/21}$ by $\exp (c_2\sqrt{n})$,
where $c_2$ depends on $c_1$.

We would like to given an outline of our proof of the Theorem.
We put $k_0=2(k+l)-\delta -1$.
Under the condition given in the Theorem,
we shall prove that, after preliminaries, i) if $n+k+l>k_0^{3/2}$, then $l\geq 0.001n$,
ii) if $n+k+l\leq k_0^{3/2}$, then $l\geq 0.001n$ (hence, $l\geq 0.001n$ in any case)
and iii) if $l\geq 0.001n$, then the conclusions of theorem holds.

\eqref{eq12} can be restated as
\begin{equation}\label{eq13}
\prod_{i_1=m+1}^k (n-i_1) \prod_{i_3=\delta +1}^l (2n+i_3)=\prod_{i_2=m+\delta +1}^{k+l} (n+i_2)
\end{equation}
or, equivalently,
\begin{equation}\label{eq14}
\prod_{i=1}^{l-\delta } \left(\frac{2n+\delta +i}{n+k+\delta +i}\right)=\prod_{j=1}^{k-m}\left(\frac{n+\delta+m+j}{n-k-1+j}\right).
\end{equation}

From \eqref{eq13}, we see that $n-i_1$ and $n+i_2$ in these products
must be composed of small prime factors.
Writing $P(\nu)$ for the largest prime factor of an integer $\nu>1$,
we can easily give an upper bound for $P(\prod_{i_1=m+1}^k (n-i_1)\prod_{i_2=m+\delta +1}^{k+l} (n+i_2))$,
as we do in Lemma \ref{lm23}.

Many results have been known concerning the largest prime factor of the product of consecutive integers.
Beginning with Sylvester's theorem \cite{Syl} that $P(\prod_{i=1}^k (x+i))>k$ for $x\geq k$,
many results concerning the multiplicative properties of $\prod_{i=1}^k (x+i)$ have been known.
Erd\H{o}s \cite{Erd} gave a more elementary proof of Sylvester's theorem.
By elementary means, Hanson \cite{Han} showed that $P(\prod_{i=1}^k (x+i))>1.5k$ for $x\geq k$
and $(x, k)\neq (2, 2), (7, 2), (5, 5)$.
Using combinatorial arguments, but with the aid of explicit estimates for $\pi(x)$,
Laishram and Shorey \cite{LS1} showed that
$P(\prod_{i=1}^k (x+i))>2k$ for $x\geq\max\{k+13, (279/262)k\}$ and $>1.97k$ for $x\geq k+13$.
With the aid of other related results such as \cite{BaBe}, \cite{LN} and \cite{LS2},
Nair and Shorey \cite{NS} shows that, if $x>4k$ and $k\geq 2$, then,
$P(\prod_{i=1}^k(x+i))>4.42k$ except only finitely many pairs $(x, k)$,
which they determined explicitly.

Methods developed in these papers allow us to prove that i) and ii), i.e.,
we cannot have $l<0.001n$.
Indeed, our argument is essentially similar to an argument from \cite{NS}.
However, in order to manipulate two products of consecutive integers,
we need a preliminary inequality given in Lemma \ref{lm22}.
Moreover, we need more complicated calculation than in papers
concerning to one product of consecutive integers.
We also need some explicit estimates for the distribution of primes
and an explicit version of Stirling's formula as in \cite{NS}.

In the case $n+k+l>k_0^{3/2}$, we shall prove an inequality involving $\pi(k_0)$ in Lemma \ref{lm31}
and then an upper bound for $k+l$ in Lemma \ref{lm32}.
With the aid of preliminary estimates, we are led to an upper bound for $n$.
Using an argument involving prime gaps and some calculation, we show that
Lemma \ref{lm23} can never hold for $n$ if $l<0.001n$.

In the case $n+k+l\leq k_0^{3/2}$, we shall prove upper and lower bounds for the size of a product
of two binomial coefficient in Lemmas \ref{lm41} and \ref{lm42} respectively.
Lemma \ref{lm41} follows from Lemma \ref{lm23}
and Lemma \ref{lm42} follows from an explicit version of Stirling's formula.
But these bounds are incompatible if $l<0.001n$.
Finally, it is relatively easy to prove iii) using known results for prime gaps.

In contrast, it seems to be difficult to obtain a general result for \eqref{eq12} in cases such as
the case $m\sim k$ but $k-m\rightarrow\infty$ and the case $l>\exp (n^A)$ with $A>1/2$.
Even specific equations such as $\binom{2n}{n}=\binom{y}{2}$ seem to be far beyond present techniques.

\section{Preliminary lemmas}

In this section, we shall introduce some preliminary lemmas.
We shall begin by some elementary lemmas.

\begin{lem}\label{lm21}
\begin{equation}\label{eq211}
(l-\delta)\log \frac{2n+l}{n+k+l}<\frac{(k-m)(k+m+\delta +1)}{n-k}
\end{equation}
and
\begin{equation}\label{eq212}
(l-\delta)\log \frac{2n}{n+k}>\frac{(k-m)(k+m+\delta +1)}{n+k+\delta }.
\end{equation}
\end{lem}

\begin{proof}
It immediately follows from \eqref{eq14} that
\begin{equation}
\left(\frac{2n+l}{n+k+l}\right)^{l-\delta}\leq \left(\frac{n+m+\delta +1}{n-k}\right)^{k-m}
\end{equation}
and therefore
\begin{equation}
\begin{split}
(l-\delta)\log \frac{2n+l}{n+k+l}\leq & (k-m)\log\frac{n+m+\delta +1}{n-k} \\
< & \frac{(k-m)(k+m+\delta +1)}{n-k}.
\end{split}
\end{equation}

Similarly,
\begin{equation}
\left(\frac{2n}{n+k}\right)^{l-\delta}\geq\left(\frac{2n+\delta +1}{n+k+\delta +1}\right)^{l-\delta}\geq \left(\frac{n+k+\delta}{n-m-1}\right)^{k-m}
\end{equation}
yields that
\begin{equation}
(l-\delta)\log \frac{2n}{n+k}\geq (k-m)\log\frac{n+k+\delta}{n-m-1}>\frac{(k-m)(k+m+\delta +1)}{n+k+\delta}.
\end{equation}
\end{proof}

This gives the following preliminary condition for $k$ and $l$.

\begin{lem}\label{lm22}
If $m\leq 0.735k$ and $l<0.001n$, then $588\leq k<0.00151n$ and $l<0.00271k$.
\end{lem}
\begin{proof}
We begin by showing that $k\geq 588$.
Indeed, if $k\leq 587$, then, using \eqref{eq211} again, we obtain
\begin{equation}
l-\delta<\frac{k^2}{(n-k)\log\frac{2.001}{1.001+\frac{k}{n}}}<1
\end{equation}
and therefore $l=\delta$, which is impossible.
Thus, we must have $k\geq 588$.

Since we have assumed that $m\leq 0.735k$, it immediately follows from \eqref{eq212} that
\begin{equation}
\frac{0.459775k^2}{(n+k+1)\log\frac{2n}{n+k}}\leq \frac{k^2-m^2}{(n+k+1)\log\frac{2n}{n+k}}<l
\end{equation}
and, recalling that $n\geq 500000$,
\begin{equation}
\frac{0.459775(k/n)^2}{(1+(k/n)+2\times 10^{-6})\log\frac{2}{1+(n/k)}}<0.001.
\end{equation}
Hence, we have $k/n<0.00151$.

Similarly, since we have assumed that $l<0.001n$, it follows from \eqref{eq211} that
\begin{equation}
l<1+\frac{(k-m)(k+m+2)}{(n-k)\log\frac{2n+l}{n+k+l}}<1+\frac{(k^2-m^2)(1+2/k)}{(n-k)\log\frac{2.001n}{1.001n+k}}.
\end{equation}
Dividing both sides by $k$, recalling the assumption that $m\leq 0.735k$ and using $k<0.00151n$,
we have
\begin{equation}
\frac{l}{k}<\frac{1}{k}+\frac{0.459775k(1+2/k)}{0.99849n\log\frac{2.001}{1.00251}}<\frac{1}{k}+0.001005\left(1+\frac{2}{k}\right).
\end{equation}
Now we know that $k\geq 588$.
Hence, we have $l/k<0.001005(1+1/294)+1/588<0.00271$.
This completes the proof of the lemma.
\end{proof}

We write $P(\nu)$ for the largest prime factor of an integer $\nu>1$.
The following lemma relates our problem to the largest prime factor
of a product of integers over two intervals.

\begin{lem}\label{lm23}
We have
\begin{equation}
P\left(\prod_{i_1=m+1}^k (n-i_1) \prod_{i_2=m_0+1}^{l+k} (n+i_2)\right)\leq k_0,
\end{equation}
where $m_0=\max\{m+\delta, \floor{l/2}\}$ and, as mentioned above, $k_0=2(l+k)-\delta -1$.
\end{lem}

\begin{rem}
A similar argument applied to $2n+i_3$ with $i_3$ odd and $\delta+1\leq i_3\leq l$
yields that the largest prime factor of the product of such integers is also $\leq k_0$.
However, this does not seem to improve our estimates.
\end{rem}

\begin{proof}
If a prime $p$ divides one of $(n-i_1)$'s with $m+1\leq i_1\leq k$, then, by \eqref{eq13},
$p$ must divide some $n+i_2$ with $m_0+1\leq i_2\leq l+k$.
Thus $p$ divides $i_1+i_2$ and therefore $p\leq 2k+l$.

Similarly, if a prime $p$ divides one of $(n+i_2)$'s with $m_0+1\leq i_2\leq l+k$,
then $p$ must divide some $n-i_1$ with $m+1\leq i_1\leq k$ or $2n+i_3$ with $\delta +1\leq i_3\leq l$.
Thus $p$ divides $i_1+i_2$ or $2i_2-i_3$.
We see that $2i_2-i_3\neq 0$ since $i_2\geq 2(m_0+1)>l\geq i_3$.
Hence, we have $p\leq 2(l+k)-(\delta +1)=k_0$.
\end{proof}

We shall use a new explicit estimate for $\pi(x)$ of Dusart in \cite[Theorem 5.1]{Dus},
although weaker estimates would still suffices.

\begin{lem}\label{lm24}
For $x>1$,
\begin{equation}
\pi(x)<\frac{x}{\log x}\left(1+\frac{1}{\log x}+\frac{2}{\log^2 x}+\frac{7.59}{\log^3 x}\right).
\end{equation}
\end{lem}

We shall also use the following version of Stirling's formula proved by Robbins \cite{Rob}.

\begin{lem}\label{lm25}
For any integer $\nu>1$,
we have $g^-(\nu)<\nu !<g^+(\nu)$,
where $g^-(z)=(z/e)^z(2\pi z)^{1/2} e^{1/12(z+1)}$ and $g^+(z)=(z/e)^z(2\pi z)^{1/2} e^{1/12z}$.
\end{lem}

\section{The first part of the Theorem - the case $n+k+l>k_0^{3/2}$}

In this section, we shall prove that $l\geq 0.001n$ under the condition
that $n+k+l>k_0^{3/2}$.

We assume that $n, m, l, k$ are integers satisfying \eqref{eq12} with
$0\leq m<k<n/2, m\leq 0.735k, l<0.001n$ and $n+k+l>k_0^{3/2}$.
As mentioned in the Introduction, we know that $n\geq 500000$.

We begin by the following estimate similar to Lemma 2.9 in \cite{NS}
and (17) in \cite{SaSh}.

\begin{lem}\label{lm31}
\begin{equation}
(n-k)^{2k+l-m-m_0-\pi(k_0)}\leq (2k+l)^{\pi(k_0)} (k-m)!(l+k-m_0)!.
\end{equation}
\end{lem}

\begin{proof}
We write $S_1$ and $S_2$ for the sets of integers, respectively,
$n-i_1$ with $m+1\leq i_1\leq k$ and $n+i_2$ with $m_0+1\leq i_2\leq l+k$ and,
for a given prime $p$, we write $v_p(\nu)$ for the exponent of a prime $p$ dividing $\nu$.

For a given prime $p$, let $n_{p, j}$ be the integer divisible by the highest power of $p$
among all integers in the interval $S_j$ for each $j=1, 2$ and $n_p$ be the integer
divisible by the highest power of $p$ among all integers in $S_1\cup S_2$.

We see that, for any prime power $p^a$, there exist at most
\[\floor{\frac{n-m-1}{p^a}}-\floor{\frac{n-k-1}{p^a}}\leq \floor{\frac{k-m}{p^a}}+1\]
integers in $S_1$ and at most 
\[\floor{\frac{n+l+k}{p^a}}-\floor{\frac{n+m_0}{p^a}}\leq \floor{\frac{l+k-m_0}{p^a}}+1\]
integers in $S_2$ which are divisible by $p^a$.

Since $v_p((k-m)!)=\sum_{a\geq 1}\floor{(k-m)/p^a}$ and
$v_p((l+k-m_0)!)= \\ \sum_{a\geq 1}\floor{(l+k-m_0)/p^a}$,
we have
\begin{equation}
v_p\left(\prod_{n-i\in S_1}(n-i)\right)\leq v_p(n_{p, 1})+v_p((k-m)!)
\end{equation}
and
\begin{equation}
v_p\left(\prod_{n+i\in S_2}(n+i)\right)\leq v_p(n_{p, 2})+v_p((l+k-m_0)!).
\end{equation}
Moreover, if $p^a$ divides both $n_{p, 1}$ and $n_{p, 2}$, then $p^a$ must divide $n_{p, 2}-n_{p, 1}$
and therefore $p^a\leq n_{p, 2}-n_{p, 1}\leq 2k+l$.
These observations lead to
\begin{equation}
\begin{split}
& v_p\left(\prod_{n-i_1\in S_1}(n-i_1) \prod_{n+i_2\in S_2}(n+i_2)\right) \\
\leq & \max\{v_p(n_{p, 1}), v_p(n_{p, 2})\}+\frac{\log(2k+l)}{\log p}+v_p((k-m)!(l+k-m_0)!) \\
= & v_p(n_p)+\frac{\log(2k+l)}{\log p}+v_p((k-m)!(l+k-m_0)!).
\end{split}
\end{equation}

Multiplication over all primes $\leq k_0$ gives
\begin{equation}
\prod_{i_1=m+1}^k (n-i_1) \prod_{i_2=m_0+1}^{l+k}(n+i_2)
\leq \left(\prod_{p\leq k_0}n_p\right) (2k+l)^{\pi(k_0)} (k-m)!(l+k-m_0)!
\end{equation}
by exploiting Lemma \ref{lm23}.
Now, omitting $n_p$'s for $p\leq k_0$ from two products in the left hand side, we conclude that
\begin{equation}
(n-k)^{2k+l-m-m_0-\pi(k_0)}\leq (2k+l)^{\pi(k_0)} (k-m)!(l+k-m_0)!.
\end{equation}
This proves the lemma.
\end{proof}

Combining Lemma \ref{lm31} with the estimate for $\pi(x)$ given in Lemma \ref{lm24},
we shall show that $k+l$ must be small.

\begin{lem}\label{lm32}
If $m\leq 0.735k, l<0.001n$ and $n+k+l>k_0^{3/2}$,
then $k+l\leq 871155$.
\end{lem}
\begin{proof}
Since $m_0\leq m+1$, Lemma \ref{lm31} yields that
\begin{equation}
\begin{split}
& \pi(2(k+l))\log((2k+l)(k_0^{3/2}-2k-l))+\log((k-m)!(l+k-m-1)!) \\
& \geq (2k+l-2m-1)\log(k_0^{3/2}-2k-l).
\end{split}
\end{equation}

Now we write $m_1=0.735k$.
By Lemma \ref{lm22}, we have $k<l+k<n-k$ and $\floor{l/2}<\floor{0.01k}<m_1$.
Thus, using Lemma \ref{lm25}, we obtain
\begin{equation}
\begin{split}
& \pi(2(k+l))\log((2k+l)(k_0^{3/2}-2k-l))+f(k-m_1)+f(l+k-m_1-1) \\
& \geq (2k+l-2m_1-1)\log(k_0^{3/2}-2k-l),
\end{split}
\end{equation}
where $f(z)=\log g^+(z)=z\log z-z+(\log(2\pi z))/2+1/(12z)$,
and therefore
\begin{equation}
\begin{split}
& \pi(2(k+l))\log(2k+l)+f(0.265k)+f(l+0.265k) \\
& -(0.53k+l-1-\pi(2(k+l)))\log(k_0^{3/2}-2k-l) \\
& \geq 0.
\end{split}
\end{equation}
We put $F=k+l$.
Since $k_0\geq 2F-2$, the above inequality yields that
\begin{equation}\label{eq31}
\begin{split}
& \pi(2F)\log(F+k)+f(0.265k)+f(F-0.735k) \\
& -(0.53k+l-1-\pi(2F))\log((2F-2)^{3/2}-F-k) \\
& \geq 0.
\end{split}
\end{equation}

Since we know that $588\leq k<F\leq 1.00271k$ from Lemma \ref{lm22},
we have
\begin{equation}
0.53k+l-1-\pi(2F)>0.53F-1-\pi(2F)>0,
\end{equation}
where the last inequality folllows from Lemma \ref{lm24}, and therefore
\begin{equation}
\begin{split}
& \frac{\pi(2F)}{F+k}+0.265f^\prime(0.265k)-0.735f^\prime(F-0.735k) \\
& +0.47\log((2F-2)^{3/2}-F-k)+\frac{F-0.47k-1-\pi(2F)}{(2F-2)^{3/2}-2k-l} \\
> & 0.265\log(0.265k)-0.735\log F+0.47\log (F^{3/2}) \\
& -\frac{0.3675}{F-0.735k}+\frac{F-0.47k-1-\pi(2F)}{(2F-2)^{3/2}} \\
\geq & 0.265\log(0.265k)-0.03\log (1.00271k)-\frac{0.3675}{0.265k}>0.
\end{split}
\end{equation}
This means that, for each fixed $F$, the left-hand side of \eqref{eq31} is increasing
with $k$ under the condition $1\leq l<0.00271k$.
Now \eqref{eq31} implies that
\begin{equation}\label{eq32}
\begin{split}
& \pi(2F)\log(2F-1)+f(0.265(F-1))+f(F-0.735(F-1)) \\
& -(0.53(F-1)-\pi(2F))\log((2F-2)^{3/2}-2F+1) \\
& \geq 0.
\end{split}
\end{equation}
With the aid of Lemma \ref{lm24}, we must have $F\leq 871155$.
\end{proof}

Now we shall prove that \eqref{eq12} cannot hold
when $0\leq m<k<n/2, m\leq 0.735k, l<0.001n$ and $n+k+l>k_0^{3/2}$.

By Lemmas \ref{lm22} and \ref{lm32}, we must have $589\leq k+l\leq 871155$ and $l<0.00271k$.
Using Lemma \ref{lm31}, we obtain $n\leq 31754673611$.

Let $d(p)$ be the gap of primes $q-p$, where $q$ is the next prime after $p$.
We know that $d(p)\leq 456$ for any prime $p\leq 3.2\times 10^{10}$ from \cite{Nic}.

By Lemma \ref{lm23}, there exist no prime in $S_2$ since $k_0<2(k+l)<n$.
Hence, if $p$ is the largest prime $p\leq n+m_0$, then
$k+l-m_0\leq d(p)-1$.
Since $n+m_0<3.2\times 10^{10}$, we have $k+l-m_0\leq 455$.
It immediately follows that
$0.265k+l\leq 455, k\leq 1713$ and $k_0\leq 2(k+l)-1\leq 3427$.

Using Lemma \ref{lm23} again, we have
$P(\prod_{m+1\leq i\leq k} (n-i))\leq 3413$.
There exists no prime among such $n-i$'s since $n-k>0.5n>3413$.
Hence, we can take two consecutive primes $q$ and $q^\prime$
satisfying $q\leq n-k-1<n-m\leq q^\prime=q+d(q)$.
We see that $q\leq n\leq 31754673611$ and $d(q)\geq 158$ since $d(q)\geq 0.265k+1>157$.

Our computer search found exactly $572960$ primes $q\leq 31754673611$ with $d(q)\geq 158$.
For all of such primes, we confirmed that
$P(\prod_{152\leq i\leq 156} (q+i))>3427\geq k_0$ and $P(\prod_{303\leq i\leq 308} (q+i))>3427\geq k_0$.
Recalling that $d(q)\leq 456$ for $q\leq 31754673611$,
we conclude that $P(\prod_{1\leq i\leq 156} (z+i))>3427$ for any positive integer $z\leq 31754673611$.
This implies that we can never have $P(\prod_{m+1\leq i\leq k} (n-i))\leq 3413$.
Thus, \eqref{eq12} cannot hold
when $0\leq m<k<n/2, m\leq 0.735k, l<0.001n$ and $n+k+l>k_0^{3/2}$.

\section{The first part of the Theorem - the case $n+k+l\leq k_0^{3/2}$}

Nextly, we shall prove that $l\geq 0.001n$ in the case $n+k+l\leq k_0^{3/2}$.
We begin by proving the following lemma, which is similar to Lemma 3 of \cite{SaSh}.

\begin{lem}\label{lm41}
If $n+k+l\leq k_0^{3/2}$, then
\begin{equation}
\binom{n-m-1}{k-m}\binom{n+k+l}{l+\frac{k}{2}-\delta}\leq (2.83)^{k_0+3k_0^{3/4}}.
\end{equation}
\end{lem}
\begin{proof}
We begin by introducing Chebyshev's functions
$\theta(x)=\sum_{p\leq x}\log p$ and $\psi(x)=\sum_{p^e\leq x}\log p$,
where $p^e$ runs all prime powers below $x$.

Lemma \ref{lm23} immediately implies that
\begin{equation}
P\left(\binom{n-m-1}{k-m}\binom{n+k+l}{l+\frac{k}{2}-\delta}\right)\leq k_0.
\end{equation}
As in \cite{Erd}, we observe that,
for any prime $p$ and integers $\nu\geq r\geq 0$ with $p^a$ dividing $\binom{\nu}{r}$, we have $p^a\leq \nu$.
Hence, we obtain
\begin{equation}
\begin{split}
\binom{n-m-1}{k-m}\binom{n+k+l}{l+\frac{k}{2}-\delta}
\leq & \left(\prod_{p\leq k_0} p\right)\left(\prod_{u\geq 2} \left(\prod_{p<(n+k+l)^{1/u}} p\right)^2 \right) \\
= & e^{\theta(k_0)+\sum_{u\geq 2}2\theta((n+k+l)^{1/u})}.
\end{split}
\end{equation}

The sum in the exponent is at most
\begin{equation}
\begin{split}
& 2\sum_{v\geq 1}\theta((n+k+l)^{1/(2v)})+2\sum_{v\geq 1}\theta((n+k+l)^{1/(2v+1)}) \\
\leq & 3\sum_{v\geq 1}\theta((n+k+l)^{1/(2v)})+\sum_{v\geq 1} \theta((n+k+l)^{1/(2v+1)}) \\
= & 3\psi((n+k+l)^{1/2})+\sum_{v\geq 1} \theta((n+k+l)^{1/(2v+1)}).
\end{split}
\end{equation}
Since we have assumed that $n+k+l\leq k_0^{3/2}$, we obtain
\begin{equation}
\sum_{v\geq 1} \theta((n+k+l)^{1/(2v+1)})<\sum_{v\geq 1} \theta(k_0^{1/(2v)})
\end{equation}
and
\begin{equation}
\begin{split}
\theta(k_0)+\sum_{u\geq 2}2\theta((n+k+l)^{1/u})<\psi(k_0)+3\psi(k_0^{3/4}).
\end{split}
\end{equation}
By Theorem 12 of \cite{RS}, we have $\psi(z)<1.03883z<z\log(2.83)$ for any real $z>0$.
This proves the lemma.
\end{proof}

We also need the other inequality, which can be derived from the explicit version
of Stirling's formula given in Lemma \ref{lm25}.

\begin{lem}\label{lm42}
If $m\leq 0.735n, l<0.001n$ and $n+k+l\leq k_0^{3/2}$, then
\begin{equation}
\log\left(\binom{n-m-1}{k-m}\binom{n+k+l}{l+k-m_0}\right)>4.6623k-2.879-\log k.
\end{equation}
\end{lem}

\begin{proof}
We write $m_1=0.735k$ as in the previous lemma.
It is clear that $m\leq \eta k<\alpha \eta n$.
Thus Lemma \ref{lm25} gives
\begin{equation}\label{eq41}
\begin{split}
\binom{n-m-1}{k-m}\geq & \frac{n-k}{n-m}\binom{n-m}{k-m} \\
= & \frac{n-k}{n-m}\frac{(n-m)!}{(n-k)!(k-m)!} \\
\geq & \frac{n-k}{n-m}\frac{g^-(n-m)}{g^+(n-k)g^+(k-m)} \\
\geq & \frac{n-k}{n-m_1}\frac{g^-(n-m_1))}{g^+(n-k)g^+(k-m_1)},
\end{split}
\end{equation}
where we note that $g^-(n-z)/((n-z)g^+(k-z))$ is a decreasing function of $z$ for $0\leq z\leq 0.735k$.

We observe that for any real $z, z_1>0$,
\begin{equation}
\begin{split}
\log \frac{(z+z_1)^{z+z_1}}{z^z z_1^{z_1}} = & \log \frac{(1+z_1/z)^{z+z_1}}{(z_1/z)^{z_1}} \\
= & z\left[\left(1+\frac{z_1}{z}\right)\log \left(1+\frac{z_1}{z}\right)-\frac{z_1}{z}\log \frac{z_1}{z}\right] \\
= & z\int_{z_1/z}^{1+z_1/z}(1+\log t) dt>z\log \frac{e z_1}{z}.
\end{split}
\end{equation}
With the aid of this inequality, we obtain
\begin{equation}
\begin{split}
\frac{g^-(n-m_1)}{g^+(n-k)g^+(k-m_1)}> & \frac{(n-m_1)^{n-m_1}}{(n-k)^{n-k} (k-m_1)^{k-m_1}}\sqrt{\frac{n-m_1}{2\pi (n-k)(k-m_1)}} \\
& \times e^{\frac{1}{12(n-m_1+1)}-\frac{1}{12(n-k)}-\frac{1}{12(k-m_1)}} \\
> & \left(e\left(\frac{n-k}{k-m_1}\right)\right)^{k-m_1}\sqrt{\frac{n-m_1}{2\pi (n-k)(k-m_1)}} \\
& \times e^{\frac{1}{12(n-m_1+1)}-\frac{1}{12(n-k)}-\frac{1}{12(k-m_1)}}.
\end{split}
\end{equation}

Thus, \eqref{eq41} becomes
\begin{equation}
\begin{split}
\log \binom{n-m-1}{k-m}\geq & \log \frac{n-k}{n-m}+(k-m_1)\left(1+\log \left(\frac{n-k}{k-m_1}\right)\right) \\
& +\frac{1}{2}\log \frac{n-m_1}{2\pi (n-k)(k-m_1)}-\frac{1}{12(n-k)}-\frac{1}{12(k-m_1)}.
\end{split}
\end{equation}
From the assumption and Lemma \ref{lm22}, we see that
$k-m_1\geq 0.265k>155$ and $n-k\geq 0.99849n\geq 499245$.
Hence, putting $\alpha=k/n$, we obtain
\begin{equation}
\begin{split}
\log \binom{n-m-1}{k-m}\geq & \log \frac{1-\alpha}{1-0.735\alpha}+0.265\alpha n\left(1+\log \left(\frac{1-\alpha}{1-0.735\alpha}\right)\right) \\
& +\frac{1}{2}\log \frac{1-0.735\alpha}{0.53\pi \alpha (1-\alpha)n}-0.0006.
\end{split}
\end{equation}
We have $\alpha\leq 0.00151$ by Lemma \ref{lm22} and therefore
\begin{equation}
\begin{split}
\log \frac{1-\alpha}{1-0.735\alpha}+\frac{1}{2}\log \frac{1-0.735\alpha}{0.53\pi(1-\alpha)}= & \frac{1}{2}\log \frac{1-\alpha}{0.53\pi (1-0.735\alpha)} \\
> & -0.2552.
\end{split}
\end{equation}
Thus, we conclude that
\begin{equation}\label{eq42}
\log \binom{n-m-1}{k-m}\geq 0.265\alpha n\left(1+\log \left(\frac{1-\alpha}{1-0.735\alpha}\right)\right)-\frac{\log(\alpha n)}{2}-0.2558.
\end{equation}

Similarly, we have
\begin{equation}
\begin{split}
\frac{g^-(n+k+l))}{g^+(k+l-m_1)g^+(n+m_1)}> & \left(\frac{e(n+m_1)}{k+l-m_1}\right)^{k+l-m_1} \\
\times \sqrt{\frac{n+k+l}{2\pi (n+m_1)(k+l-m_1)}} & \times e^{\frac{1}{12(n+k+l+1)}-\frac{1}{12(n+m_1)}-\frac{1}{12(k+l-m_1)}}
\end{split}
\end{equation}
and, noting that $m_0\leq \max\{m+1, l/2\}\leq \max\{m_1+1, 0.002k\}=m_1+1$,
\begin{equation}
\begin{split}
& \log \binom{n+k+l}{k+l-m_0}\geq \log \frac{k+l-m_1}{n+m_1+1}+(k+l-m_1)\left(1+\log\frac{n+m_1}{k+l-m_1}\right) \\
& \qquad +\frac{1}{2}\log\frac{n+k+l}{2\pi(n+m_1)(k+l-m_1)}-\frac{1}{12(n+m_1)}-\frac{1}{12(k+l-m_1)} \\
& \geq \log \frac{(0.265+\lambda)\alpha}{1+0.735\alpha}+(0.265+\lambda)\alpha n\left(1+\log\frac{1+0.735\alpha}{(0.265+\lambda)\alpha}\right) \\
& \qquad +\frac{1}{2}\log \frac{1+\alpha(1+\lambda)}{2\pi(0.265+\lambda)(1+0.735\alpha)\alpha n}-0.0006,
\end{split}
\end{equation}
where we put $\lambda=l/k$ and used the fact that $k-m_1\geq 0.265k>156$ and $n-k\geq 499245$.

We have $\alpha\leq 0.00151$ and $\lambda\leq 0.00271$ by Lemma \ref{lm22} and therefore
\begin{equation}
\begin{split}
& \log \frac{(0.265+\lambda)\alpha}{1+0.735\alpha}+\frac{1}{2}\log \frac{1+\alpha(1+\lambda)}{2\pi(0.265+\lambda)(1+0.735\alpha)} \\
& =\frac{1}{2}\log \frac{(1+\alpha(1+\lambda))(0.265+\lambda)}{2\pi(1+0.735\alpha)^3} \\
& \geq \frac{1}{2}\log \frac{0.265(1+\alpha)}{2\pi(1+0.735\alpha)^3}>-1.5788.
\end{split}
\end{equation}
Thus, we conclude that
\begin{equation}\label{eq43}
\begin{split}
\log \binom{n+k+l}{k+l-m_0}\geq & (0.265+\lambda)\alpha n\left(1+\log\frac{1+0.735\alpha}{(0.265+\lambda)\alpha}\right) \\
& +\frac{\log (\alpha/n)}{2}-1.5794.
\end{split}
\end{equation}

Combining \eqref{eq42} and \eqref{eq43}, we obtain
\begin{equation}\label{eq36}
\begin{split}
& \log\left(\binom{n-m-1}{k-m}\binom{n+k+l}{l+k-m_0}\right) \\
> & 0.265\alpha n\left(1+\frac{1-\alpha}{0.265\alpha}\right)+(0.265+\lambda)\alpha n\left(1+\frac{1+0.735\alpha}{(0.265+\lambda)\alpha}\right) \\
 & -1.8352-\log n.
\end{split}
\end{equation}

We put
\begin{equation}
h(\alpha, \lambda)=0.265\left(1+\log\frac{1-\alpha}{0.265\alpha}\right)+\left(0.265+\lambda\right)\left(1+\log\frac{1+0.735\alpha}{\alpha\left(0.265+\lambda\right)}\right).
\end{equation}
For a fixed $\lambda$, $h(\alpha, \lambda)$ is decreasing over $0<\alpha<1$ since $\partial h/\partial\alpha=-0.265/\alpha(1-\alpha)-(0.265+\lambda)/\alpha(1+0.735\alpha)<0$.
On the other hand, for a fixed $\alpha$, $h(\alpha, \lambda)$ is increasing over $0<\lambda<0.00271$ since $\partial h/\partial \lambda(\alpha, \lambda)=\log (1+0.735\alpha)/(\alpha(\lambda+0.265))>0$.
Hence, $h(\alpha, \lambda)\geq h(0.00151, 0)>4.6623$.
Now \eqref{eq36} gives
\begin{equation}
\log\left(\binom{n-m-1}{k-m}\binom{n+k+l}{l+k-m_0}\right)>4.6623\alpha n-1.8352-\log n.
\end{equation}
Since we have assumed that $n+k+l\leq k_0^{3/2}=(2(k+l)-\delta-1)^{3/2}$
and Lemma \ref{lm22} gives $l<0.00271k$, we have $n<(2(k+l))^{3/2}<(2.00542k)^{3/2}$.
Hence, $1.8352+\log n<2.879+(3/2)\log k$.
This proves the lemma.
\end{proof}

Now we shall prove that \eqref{eq12} cannot hold
when $0\leq m<k<n/2, m\leq 0.735k, l<0.001n$ and $n+k+l\leq k_0^{3/2}$.
We apply Lemma \ref{lm41} to obtain
\begin{equation}
\log\left(\binom{n-m-1}{k-m}\binom{n+k+l}{l+k-m_0}\right)\leq {k_0+3k_0^{3/4}}\log (2.83)
<1.0433k+3.13k^{3/4},
\end{equation}
where we used Lemma \ref{lm22} to see that $k_0\leq k+l<1.00271k$.
Now, using Lemma \ref{lm42}, we have
\begin{equation}
4.6623k-1.8344-\log k<1.0433k+3.13k^{3/4},
\end{equation}
which is impossible for $k\geq 588$.
Thus, we see that \eqref{eq12} can never hold
when $0\leq m<k<n/2, m\leq 0.735k$ and $l<0.001n$.

\section{The remaining case: $l\geq 0.001n$}

In this section, we discuss the remaining case:
$n, m, l, k$ are integers satisfying \eqref{eq12} with 
$0\leq m<k<n/2, m\leq 0.735k$ and $l\geq 0.001n$.
Since $n\geq 500000$, Proposition 5.4 of \cite{Dus} implies that
there exists at least one prime $2n+\delta+1\leq p\leq 2n+l$.

Now, let $2n+t$ be the largest prime $\leq 2n+l$.
From \eqref{eq13}, it is clear that $2n+t$
must divide $n+i$ for some $i\leq k+l$.
Using Proposition 5.4 of \cite{Dus} again, we must have $n-k\leq l-t<(2n+l)/\log^3 (2n+l)$
and therefore
$\log \binom{2n+l}{n-k}<(n-k)\log (2n+l)<(2n+l)/\log^2 (2n+l)$.

On the other hand, since $m\leq 0.735 k<0.735 n$, we have
\begin{equation}
\begin{split}
\binom{2n+\delta}{n-m}
> & \frac{g^-(2n)}{g^+(0.735 n)g^+(1.265n)} \\
> & \left(\frac{(2/0.735)^2}{((2/0.735)-1)^{1.265}}\right)^n\sqrt{\frac{1}{0.929775\pi n}} \\
& \times e^{\frac{1}{12(2n+1)}-\frac{1}{12n}\left(\frac{1}{\eta}+\frac{1}{2-\eta}\right)}
\end{split}
\end{equation}
by Lemma \ref{lm25}.
Taking its logarithm, we have
\begin{equation}
\frac{2n+l}{\log^2 (2n+l)}>1.3132n-\frac{1}{2}\log n-0.5359.
\end{equation}
Observing that $\log^2 (2n)((1/2)\log n+0.5359)<0.00271n$ for $n>500000$,
we conclude that
\begin{equation}
l>n(1.3132\log^2(2n)-2.00271).
\end{equation}

Furthermore, from the result of \cite{BHP}, we must have
$n-k\leq l-t<(2n+l)^{21/40}$ for sufficiently large $n$.
Proceeding as above, we obtain
\begin{equation}
(2n+l)^{21/40}\log (2n+l)>1.3132n-\frac{1}{2}\log n-0.5359.
\end{equation}
Thus, we conclude that $l>(cn/\log n)^{40/21}$ for sufficiently large $n$.
This completes the proof of the Theorem.

{}
\end{document}